\documentclass[11pt,reqno]{amsart}

\usepackage[dvips]{graphicx}
\usepackage{amssymb, latexsym, amsmath, amscd, array, amsthm, epsfig, pstricks}
\usepackage[all]{xy}
\usepackage{bbm}

\usepackage{tikz-cd}

\usepackage[pagebackref=true]{hyperref}

\usepackage{breakurl}   

\usepackage{graphicx}

\newcommand{\R}{{\mathbb R}}
\newcommand{\hyp}{{\mathbb H}}

\newcommand{\HH}{\mathcal{H}}

\newcommand{\UR}{{\rm UR}}
\newcommand{\Isom}{{\rm Isom}}

\newcommand{\diam}{{\rm diam}}

\newcommand{\scal}{{\rm scal}}

\newcommand{\sys}{{\rm sys}}

\newcommand{\vol}{{\rm vol}}

\newcommand{\id}{{\rm id}}
\newcommand{\length}{{\rm length}}

\newcommand{\ie}{{\it i.e.}}
\newcommand{\eg}{{\it e.g.}}

\numberwithin{equation}{section}

\newtheorem{theorem}{Theorem}[section]
\newtheorem{proposition}[theorem]{Proposition}
\newtheorem{corollary}[theorem]{Corollary}
\newtheorem{lemma}[theorem]{Lemma}

\theoremstyle{definition}
\newtheorem{definition}[theorem]{Definition}
\newtheorem{example}[theorem]{Example}
\newtheorem{remark}[theorem]{Remark}

\newtheorem{conjecture}[theorem]{Conjecture}

\long\def\forget#1\forgotten{} 

\begin{document}

\title{Macroscopic scalar curvature and local collapsing}


\author[S.~Sabourau]{St\'ephane Sabourau}

\thanks{Partially supported by the ANR project Min-Max (ANR-19-CE40-0014).}

\address{\parbox{\linewidth}{LAMA, Univ Paris Est Creteil, Univ Gustave Eiffel, CNRS, F-94010, Cr\'eteil, France \\
CRM (UMI3457), CNRS, Montr\'eal, QC H3C 3J7, Canada}}

\email{stephane.sabourau@u-pec.fr}



\begin{abstract}
Consider a closed Riemannian $n$-manifold~$M$ admitting a negatively curved Riemannian metric.
We show that for every Riemannian metric on~$M$ of sufficiently small volume, there is a point in the universal cover of~$M$ such that the volume of every ball of radius~$r \geq 1$ centered at this point is greater or equal to the volume of the ball of the same radius in the hyperbolic $n$-space.
We also give an interpretation of this result in terms of macroscopic scalar curvature.
This result, which holds more generally in the context of polyhedral length spaces, is related to a question of Guth.
Its proof relies on a generalization of recent progress in metric geometry about the Alexandrov/Urysohn width involving the volume of balls of radius in a certain range with collapsing at different scales.
\end{abstract}

\maketitle


\section{Introduction}

The scalar curvature of a closed Riemannian $n$-manifold~$M$ describes how the volume of infinitesimal balls in~$M$ compares to the volume of infinitesimal balls in the Euclidean $n$-space.
More precisely, the volume expansion of a ball of radius~$r$ centered at $x \in M$ satisfies
\begin{equation} \label{eq:scal}
\vol(B(x,r)) = \omega_n \, r^n \left( 1 - \frac{\scal(M,x)}{6(n+2)} \, r^2 + O(r^3) \right)
\end{equation}
as $r$ goes to zero, where $\scal(M,x)$ is the scalar curvature of~$M$ at~$x$ and $\omega_n$ is the volume of a unit ball in the Euclidean $n$-space; see~\cite[Theorem~3.98]{GHL}.
Understanding the relationship between scalar curvature and the topology of~$M$ is a major problem in Riemannian geometry.

\medskip

In this article, we will be interested in macroscopic versions of the following conjecture attributed to Schoen (which follows from a conjecture of Schoen about the Yamabe invariant of hyperbolic manifolds); see~\cite{schoen} and~\cite{Guth11}.
This conjecture was also stated by Gromov~\cite[3.A]{gro86} with nonsharp constants.

\begin{conjecture}[Schoen] \label{conj}
Let $(M,{\rm hyp})$ be a closed hyperbolic $n$-manifold and let $g$ be another Riemannian metric on~$M$.
If $\scal(g,x) \geq \scal({\rm hyp})$ for every $x \in M$ then 
\[
\vol(M,g) \geq \vol(M,{\rm hyp}).
\]

Equivalently, using~\eqref{eq:scal}, if $\vol(M,g) < \vol(M,{\rm hyp})$ then there exists $x_0 \in M$ such that
\[
\vol_g(B(x_0,r)) > \vol_{\rm hyp}(B(r))
\]
for every $r>0$ small enough.
\end{conjecture}

This conjecture is true in dimension~$2$ by the Gauss-Bonnet formula and in dimension~$3$ from Perelman's work; see~\cite[Proposition ~93.10]{KL08}.
In higher dimension, it also holds true for Riemannian metrics close enough to the hyperbolic one, see~\cite[Corollaire~C]{BCG91}, or if one replaces scalar curvature with Ricci curvature, see~\cite{BCG}.

\medskip

Following~\cite{GuthICM}, this leads us to introduce the following notion.
The \emph{macroscopic scalar curvature} of a closed Riemannian $n$-manifold~$M$ at scale~$r$ at~$x \in M$, denoted by $\scal_r(M,x)$, is defined as the unique real~$s$ such that
\[
\vol(B_{\tilde{M}}(\tilde{x},r)) = \vol(B_{\hyp_s^n}(r))
\]
where $\tilde{x}$ is a lift of~$x$ in the universal cover~$\tilde{M}$ of~$M$ and~$\hyp_s^n$ is the simply-connected $n$-dimensional space form with constant curvature~$s$.
It is more conveniently characterized as follows
\[
\scal_r(M,x) \leq s \mbox{ if and only if } \vol(B_{\tilde{M}}(\tilde{x},r)) \geq \vol(B_{\hyp_s^n}(r)).
\]
For example, the macroscopic scalar curvature of a flat torus at any scale is zero.
Note that this property fails if one does not take balls in the universal cover of~$M$, but only in~$M$, in the definition of the macroscopic scalar curvature, as otherwise, it would be positive at a large enough scale.
By~\eqref{eq:scal}, at infinitesimally small scale, we have
\[
\lim_{r \to 0} \scal_r(M,x) = \scal(M,x).
\]

In a different direction, the macroscopic scalar curvature at large enough scale provides information on the exponential growth rate of the volume of balls in the universal cover of~$M$, also known as the volume entropy, a much-studied geometric invariant related to the growth of the fundamental group and the dynamics of the geodesic flow.
This leads us to define
\[
V_{\tilde{M}}(r) = \sup_{\tilde{x} \in \tilde{M}} \vol(B(\tilde{x},r))
\]
as the maximal volume of a ball of radius~$r$ in the universal cover of~$M$.
As explained in~\cite{Guth11} and~\cite{GuthICM}, the celebrated theorem of Besson, Courtois and Gallot~\cite{BCG} on the minimal volume entropy provides a macroscopic version of Schoen's conjecture~\ref{conj} at large enough scales.
Stated in a way suited for comparison (ignoring its rigidity counterpart), this result takes the following form.

\begin{theorem}[Besson-Courtois-Gallot~\cite{BCG}] \label{theo:BCG}
Let $(M,{\rm hyp})$ be a closed hyperbolic $n$-manifold and let $g$ be another Riemannian metric on~$M$.
If $\vol(M,g) < \vol(M,{\rm hyp})$ then there exists~$r_0>0$ (depending on~$g$) such that for every $r \geq r_0$
\[
V_{\tilde{M}}(r) > V_{\hyp^n}(r).
\]


In particular, if $\scal_r(g,x) > \scal_r({\rm hyp})$ for every $r$ large enough and every $x \in M$ then $\vol(M,g) > \vol(M,{\rm hyp})$.
\end{theorem}

A version of this theorem was first established by A.~Katok~\cite{katok} in dimension~$2$ and a nonsharp version was obtained before by Gromov~\cite{gro82} in every dimension.

\medskip

In~\cite{Guth11}, Guth asks for an estimate on~$r_0$ after proving the following nonsharp macroscopic version of Schoen's conjecture.

\begin{theorem}[Guth~\cite{Guth11}] \label{theo:guth}
Let $(M,{\rm hyp})$ be a closed hyperbolic $n$-manifold and let $g$ be another Riemannian metric on~$M$.
Then, for every $r \geq 1$, there exists a constant $\delta_{n,r}>0$ depending only on~$n$ and~$r$, such that if $\vol(M,g) \leq \delta_{n,r} \, \vol(M,{\rm hyp})$ then 
\[
V_{\tilde{M}}(r) \geq V_{\hyp^n}(r).
\]


In other words, if $\scal_r(g,x) \geq \scal_r({\rm hyp})$ for every $x \in M$ then $\vol(M,g) \geq \delta_{n,r} \, \vol(M,{\rm hyp})$.
\end{theorem}

Further volume lower bounds have recently been obtained by Alpert and Funano for hypersurfaces in closed manifolds with macroscopic scalar curvature bounded below as a consequence of this result; see~\cite{AF17} and~\cite{A}.

\medskip

Theorem~\ref{theo:guth} gives relatively better volume estimates for unit balls (that is, for $r=1$) than for balls of large radius as the constant~$\delta_{n,r}$ falls off exponentially or faster with~$r$.
In~\cite{Guth11}, Guth suggests that one could try to combine the approaches of~\cite{gro82}, \cite{BCG} and~\cite{Guth11} to obtain a uniform volume estimate with $\delta_{n,r} = \delta_n$ depending only on~$n$.
Such a uniform bound was obtained for surfaces by Karam~\cite{K15}.
In higher dimension, Balacheff and Karam~\cite{BK19} proved a similar result for negatively curved metrics~$g$ using techniques developed in~\cite{gro82}.

\medskip

In this article, we establish the following result in this direction, following a different approach.
See Theorem~\ref{theo:ball}, Corollary~\ref{coro:ball} and Corollary~\ref{coro:final} for more general statements.

\begin{theorem} \label{theo:hyp}
Let $(M,{\rm hyp})$ be a closed hyperbolic $n$-manifold and let $g$ be another Riemannian metric on~$M$.
Then, there exists a constant $\delta'_{n}>0$ depending only on~$n$, such that if \mbox{$\vol(M,g) \leq \delta'_{n}$} then
\[
V_{\tilde{M}}(r) \geq V_{\hyp^n}(r)
\]
for every $r \geq 1$.


In other words, if $\scal_r(g,x) \geq \scal_r({\rm hyp})$ for some $r \geq 1$ and every $x \in M$ then \mbox{$\vol(M,g) \geq \delta_{n}'$.}
\end{theorem}

This result provides a uniform estimate on the radius of balls in the universal cover of~$M$ whose volume is at least the volume of balls of the same radius in~$\hyp^n$.
Furthermore, the balls of radius~$r \geq 1$ of large volume in the universal cover of~$M$ can be assumed to be centered around the same point; see Corollary~\ref{coro:final}.
Actually, we obtain a more general lower bound on the maximal volume of balls that holds for every~$r \geq 0$; see Corollary~\ref{coro:ball}.
However, contrary to Theorem~\ref{theo:guth} and Theorem~\ref{theo:BCG}, the condition of the volume of~$M$ which guarantees the existence of balls of large volume in the universal cover of~$M$ does not involve the volume of the hyperbolic metric in~$M$.
This leads to a stronger condition for manifolds of large hyperbolic volume.
Note that similar versions of Theorem~\ref{theo:hyp} hold for more general spaces, including closed manifolds admitting negatively curved metrics.
See Corollary~\ref{coro:ball} and Corollary~\ref{coro:final}.

\medskip

The overall strategy of the proof of Theorem~\ref{theo:hyp} is inspired by the approach developed in~\cite{S} and extended in~\cite{BS} to prove that the minimal volume entropy of a closed manifold admitting a hyperbolic metric is positive.
(As explained above, this result was first obtained in~\cite{gro82} using bounded cohomology argument, before a sharp version was established in~\cite{BCG} using the so-called barycenter map.)
In~\cite{S}, we show that for every closed Riemannian $n$-manifold admitting a hyperbolic metric, there exist two loops~$\gamma_1$ and~$\gamma_2$ based at the same point, of length $\lesssim \vol(M,g)^{\frac{1}{n}}$, whose homotopy classes generate a subgroup~$\Gamma \leqslant \pi_1(M)$ of positive algebraic entropy.
The proof of this estimate relies on filling techniques developed in~\cite{gro83} to establish systolic inequalities, and more precisely, on a lower bound on the filling radius.
A lower bound on the volume entropy of~$M$ immediately follows from this estimate by standard comparison argument.
Now, in order to derive a lower bound of the volume of balls in the universal cover of~$M$ for any radius, and not simply for asymptotically large radius, we need to show that there is sufficiently volume around the basepoint of the two loops~$\gamma_1$ and~$\gamma_2$ to distribute it in the universal cover of~$M$ under the action of~$\Gamma$.
The approach developed in~\cite{S}, based on filling techniques, does not readily provide any information on where the volume is located and whether there is enough volume around the basepoint of the two loops.
In a different direction, Guth~\cite{Guth11} shows that the volume in the universal cover of a closed aspherical Riemannian manifold is not too diffuse and that a non-negligeable part of it is contained in a ball of given radius, but this result does not say where such a ball is located\footnote{Though we believe we can combine the argument of~\cite{S} with the construction in~\cite{Guth11} to obtain the desired result, we will use different, more elementary, techniques to reach the same conclusion.
These techniques also have the advantage to apply to more general spaces than closed manifolds.}.
Later, building upon on the construction of~\cite{Guth11}, he extended this result into the following theorem.

\begin{theorem}[Guth~\cite{Guth17}] \label{theo:guth2}
Let $M$ be a closed Riemannian $n$-manifold.
If every ball of radius~$R$ in~$M$ has volume at most~$\varepsilon_n R^n$ for a sufficiently small constant~$\varepsilon_n >0$ then the Alexandrov width of~$M$ is at most~$R$.
That is, there exists a continuous map from~$M$ to a simplicial $(n-1)$-complex~$P$ whose fibers have radius at most~$R$.
\end{theorem}

This theorem was generalized by Liokumovich, Lishak, Nabutovsky and Rotman~\cite{LLNR} to metric spaces and Hausdorff content instead of volume.
The proof of this result was greatly simplified by Papasoglu~\cite{P} using the minimal hypersurface approach of Schoen and Yau~\cite{SY79} developed by Guth~\cite{Guth10} in a similar context.
Recently, Nabutovsky~\cite{nab} further extended this approach improving the constant from an exponential bound to a linear one.

\medskip

Now, one can show that the Alexandrov width of a closed Riemannian $n$-manifold~$M$ admitting a hyperbolic metric is bounded from below (up to a multiplicative constant depending only on~$n$) by the minimal length~$L$ such that there exist two loops~$\gamma_1$ and~$\gamma_2$ based at the same point, of length at most~$L$, whose homotopy classes generate a free subgroup of~$\pi_1(M)$; see Section~\ref{sec:margulis}.
This estimate can be interpreted as a bound on Margulis' constant; see Definition~\ref{def:margulis}.
Thus, by Theorem~\ref{theo:guth2}, there exists a ball of radius $\frac{1}{2} L$ in~$M$ of volume at least $c_n L^n$.
As explained above, we would like to show that such a ball~$B$ is centered at the basepoint of the two loops~$\gamma_1$ and~$\gamma_2$ in order to derive an exponential lower bound on the volume of balls in the universal cover~$\tilde{M}$ by taking the translates of~$B$ in~$\tilde{M}$ under the free subgroup induced by the two loops in homotopy.
This does not follow directly from the Alexandrov/Urysohn width estimates of \cite{Guth17}, \cite{LLNR}, \cite{P} or~\cite{nab}.
Instead, we need to extend Theorem~\ref{theo:guth2} by allowing the fibers in the definition of the Alexandrov width to have diameters of various size so that they capture the thick part of~$M$ in an optimal way.

\medskip

This leads us to introduce the following definition.


\begin{definition} \label{def:collapse}
Recall that a compact $n$-polyhedral length space is a finite simplicial $n$-complex endowed with an intrinsic metric (or length structure).
Let $X$ be a compact $n$-polyhedral length space and $\rho:X \to (0,\infty)$ be a function.
The space~$X$ \emph{$\rho$-collapses} if there exists a continuous map $\pi:X \to P$ to a finite simplicial $(n-1)$-complex such that every fiber of~$\pi$ lies in a ball $B(x,\rho(x))$ for some $x \in X$.
\end{definition}

Note that the Alexandrov width of a space~$X$ that $\rho$-collapses is bounded by~$\sup \rho$.
Thus, the bound on the size of the fibers of $\pi:X \to P$ in Definition~\ref{def:collapse} better reflects the shape of~$M$ than the corresponding uniform bound on the size of the fibers of the map in the definition of the Alexandrov width; see Theorem~\ref{theo:guth2}.
The collapsing of~$X$ can also be interpreted in terms of coverings; see Proposition~\ref{prop:mult}.

\medskip

We have the following general theorem which extends Theorem~\ref{theo:guth2} and some of its generalizations in~\cite{LLNR}, \cite{P} or~\cite{nab}.
The point of this theorem is that it allows us to have a better control on where the space~$X$ can collapse or not.

\begin{theorem} \label{theo:0}
Let $X$ be a compact $n$-polyhedral length space and $\rho:X \to (0,\infty)$ be a continuous function.
Suppose that the space~$X$ does not $\rho$-collapse.
Then there exists $x_0 \in X$ such that
\[
\HH_n(B_X(x_0,r)) \geq c_n \, r^n
\]
for every $r \in [0,\rho(x_0)]$, where $c_n$ is an explicit positive constant depending only on~$n$.
\end{theorem}

The proof of this theorem relies on the general approach developed by Papasoglu~\cite{P} and recently further improved by Nabutovsky~\cite{nab}.
As explained above, this approach can be traced back to works of Guth~\cite{Guth10} and, before, of Schoen and Yau~\cite{SY79}.
We will follow the arguments of~\cite{nab}.

\medskip

Now, if~$\rho$ is given by the Margulis function, see Definition~\ref{def:margulis}, and the space~$X$ does not $\rho$-collapse, we can show that, under some condition of the fundamental group of~$X$, there exist two loops of~$X$ based at~$x_0$, of length~$\lesssim \rho(x_0)$, generating a free subgroup in~$\pi_1(X)$ as desired.

\medskip

The article is organized as follows.
In Section~\ref{sec:collapse}, we prove Theorem~\ref{theo:0}.
Then, in Section~\ref{sec:margulis}, we introduce the Margulis function and derive a general version of Theorem~\ref{theo:hyp}.



\section{Collapse and volume localization} \label{sec:collapse}

The goal of this section is to give an interpretation of the collapse of a metric space~$X$ in terms of coverings through the following classical result and to prove Theorem~\ref{theo:0} and related results.
Given a subset~$A \subseteq X$, we denote by~$\bar{A}$ its closure in~$X$.

\begin{proposition} \label{prop:mult}
Let $X$ be a compact $n$-polyhedral length space and $\rho:X \to (0,\infty)$ be a function.
The space~$X$ $\rho$-collapses if and only if there exists a finite covering of~$X$ of multiplicity at most~$n$ by open subsets~$U_i$ whose closures~$\bar{U}_i$ are covered by balls $B(x_i,\rho(x_i))$ with $x_i \in X$.
\end{proposition}

\begin{remark}
The statement of the proposition would still hold by requiring only~$U_i$ to be covered by a ball $B(x_i,\rho(x_i))$.
Though more technical, the current statement simplifies the proofs of forthcoming results, in particular Lemma~\ref{lem:XY}.
\end{remark}

\begin{proof}[Proof of Proposition~\ref{prop:mult}]
Let $\pi:X \to P$ be a continuous map to a finite simplicial $(n-1)$-complex~$P$ such that every fiber of~$\pi$ lies in a ball $B(x,\rho(x))$ for some $x \in X$.
Since $P$ is a finite simplicial $(n-1)$-complex, the finite covering of~$P$ formed of the open stars~${\rm st}(v) \subseteq P$ of the vertices~$v$ of~$P$ has multiplicity at most~$n$.
Subdividing~$P$ if necessary, we can assume that the preimages $U_v = \pi^{-1}({\rm st}(v)) \subseteq X$ of these open stars have their closures~$\bar{U}_v$ lying in the same balls $B(x,\rho(x))$ covering the compact fibers~$\pi^{-1}(v)$.
Furthermore, the subsets~$U_v$ form a finite covering of~$X$ of multiplicity at most~$n$.

Let $\mathcal{U}=\{ U_i \mid i=1,\dots,m \}$ be a finite covering of~$X$ of multiplicity at most~$n$ by open subsets~$U_i$ with $\bar{U}_i \subseteq B(x_i,\rho(x_i))$ for some $x_i \in X$.
Take a partition of unity~$\{ \phi_i \}$ of~$X$ subordinated to~$\{ U_i \}$.
Consider the map $\pi:X \to \Delta^{m-1}$ defined by
\[
\pi(x)=(\phi_1(x),\dots,\phi_m(x))
\]
in the barycentric coordinates of~$\Delta^{m-1}$.
The nerve~$N$ of the covering~$\mathcal{U}$ is a finite simplicial complex with one vertex~$v_i$ for each open set~$U_i$, where $v_{i_0},\dots,v_{i_k}$ span a $k$-simplex of~$N$ if and only if the intersection $\cap_{j=1}^k U_{i_j}$ is nonempty.
By construction, the dimension of the nerve~$N$ is one less than the multiplicity of the covering~$\mathcal{U}$.
That is, $\dim N \leq n-1$.
Now, we identify in a natural way the vertices~$\{ v_i \}$ of~$N$ with the vertices of~$\Delta^{m-1}$.
With this identification, the nerve~$N$ of~$X$ lies in~$\Delta^{m-1}$.
Furthermore, the image of~$\pi$ lies in~$N$ and every fiber of~$\pi:X \to N$ lies in one of the open subsets~$U_i \subseteq B(x_i,\rho(x_i))$.
Hence, the space~$X$ $\rho$-collapses.
\end{proof}

Let us recall Eilenberg's inequality; see~\cite[13.3.1]{BZ}.

\begin{theorem}[Eilenberg's inequality]
Let $f:X \to Y$ be a Lipschitz map between separable metric space (\ie, compact metric spaces).
Then for every $A \subseteq X$ and every $0 \leq m \leq n$, we have
\begin{equation} \label{eq:eilenberg}
\int_Y^{\mkern-20mu *} \HH_{n-m}(A \cap f^{-1}(y)) \, d\HH_m(y) \leq \frac{v_{n-m} v_m}{v_n} \, ({\rm Lip} \, f)^m \, \HH_n(A)
\end{equation}
where $\int^{\mkern-18mu *}$ denotes the upper Lebesgue integral, $v_n$ is the volume of the unit ball in~$\R^n$ and $\HH_n$ is the $n$-dimensional Hausdorff measure.
\end{theorem}

The following result is the equivalent to~\cite[Lemma~2.1]{nab} in our language.

\begin{lemma} \label{lem:n=1}
Let $X$ be a compact $1$-polyhedral length space and $\rho:X \to (0,\infty)$ be a function.
Suppose that for every $x \in X$, there exists a ball $B(x,r_x) \subseteq X$ of radius $r_x \in [0,\rho(x)]$ such that $\HH_1(B(x,r_x)) < r_x$.
Then the space~$X$ $\rho$-collapses.
\end{lemma}

\begin{proof}
By the Eilenberg inequality~\eqref{eq:eilenberg} with $m=n=1$, for every $x \in X$, there exists $\tau_x \in (0,r_x)$ such that the sphere $S(x,\tau_x) \subseteq X$ is empty.
Thus, every connected component of~$X$ lies in a ball $B(x,\rho(x))$ for some $x \in X$.
\end{proof}

The proof of Theorem~\ref{theo:x_0} rests on the following notion, which generalizes the definition given in~\cite[Definition~2.3]{nab}; see also~\cite{P}.

\begin{definition}
Let $X$ be a compact $n$-polyhedral length space and $\rho:X \to (0,\infty)$ be a bounded function.
A compact $(n-1)$-subcomplex~$Y$ of~$X$ \emph{$\rho$-separates}~$X$ if for every connected component~$C$ of~$X \setminus Y$, there exists $x \in X$ (depending on~$C$) such that $\bar{C} \subseteq B_X(x,\rho(x))$.
\end{definition}

The following result is a mere adaptation of~\cite[Lemma~2.5]{nab}.

\begin{lemma} \label{lem:XY}
Let $X$ be a compact $n$-polyhedral length space and $\rho:X \to (0,\infty)$ be a function.
Let $Y$ be a compact $(n-1)$-subpolyhedron of~$X$ which $\rho$-separates~$X$.
Denote by $\rho_Y$ the restriction of~$\rho$ to~$Y$.
Suppose that the space~$Y$ $\rho_Y$-collapses.
Then the space~$X$ $\rho$-collapses.
\end{lemma}

\begin{proof}
We follow the argument presented in~\cite[Lemma~2.5]{nab}. 
Recall that the polyhedral space~$Y$ is endowed with the length structure induced from~$X$.
Let $\mathcal{U}=\{U_i\}$ be a finite covering of~$Y$ of multiplicity at most~$n-1$ by open subsets~$U_i$ of~$Y$ with $\bar{U}_i \subseteq B_Y(y_i,\rho(y_i))$ for some $y_i \in Y$.
We can thicken every subset~$U_i$ into an open subset~$V_i$ of~$X$ with $\bar{V}_i \subseteq B_Y(y_i,\rho(y_i))$, without increasing the multiplicity.
(Note that $B_Y(y,r) \subseteq B_X(y,r)$ for every $y \in Y$.)
Adding the connected components of~$X \setminus Y$ to the family~$\{ V_i\}$, we obtain a covering of~$X$ of multiplicity at most~$n$ (one more than the multiplicity of~$\mathcal{U}$) by open subsets of~$X$ whose closures lies in balls $B_X(x,\rho(x))$ with $x \in X$.
Hence, the space~$X$ $\rho$-collapses.
\end{proof}

Denote by $\lambda_n = \frac{v_n}{2 v_{n-1}}$ the coefficient in Eilenberg's inequality with $m=1$; see~\eqref{eq:eilenberg}.
Note that $\lambda_1=1$.

\begin{theorem} \label{theo:x_0}
Let $X$ be a compact $n$-polyhedral length space and $\rho:X \to (0,\infty)$ be a continuous function.
Suppose that the space~$X$ does not $\rho$-collapse.
Then there exists $x_0 \in X$ such that
\[
\HH_n(B_X(x_0,r)) \geq c_n \, r^n
\]
for every $r \in [0,\rho(x_0)]$, where $c_n = \frac{1}{n!} \, \prod_{i=1}^n \lambda_i$.
\end{theorem}

\begin{proof}
Let $\eta>0$ be small enough.
We can $\eta$-approximate the distance function~$d_X(x_0,\cdot)$ to~$x_0$ by a piecewise smooth (or even piecewise linear, after taking sufficiently fine subdivisions) $(1+\eta)$-Lipschitz function whose level sets are $(n-1)$-subpolyhedra.
To avoid burdening the argument with standard approximations, we will assume that the spheres~$S_X(x_0,r)$ are $(n-1)$-subpolyhedra, keeping in mind that the inequalities below, including the Eilenberg inequality, only hold with an extra $1+o(\eta)$ factor with $\eta>0$ arbitrarily small.
Letting $\eta$ go to zero, this will not change the final estimate.

\medskip

Let us show by induction on the dimension~$n$ of~$X$ that for every $\varepsilon \in (0,1)$ and every $\alpha \in (0,1)$, there exists $x_0 \in X$ such that
\begin{equation} \label{eq:lowerX}
\HH_n(B_X(x_0,r)) \geq \frac{1-\varepsilon}{n!} \left( {\textstyle \prod\limits_{i=1}^{n}} \lambda_i \right) r^n
\end{equation}
for every $r \in [\alpha  \, \rho(x_0),\rho(x_0)]$.
The case~$n=1$ with $c_1=1$ is covered by Lemma~\ref{lem:n=1}.

In the general case~$n \geq 2$, fix $\varepsilon \in (0,1)$ and $\alpha \in (0,1)$.
Denote $\varepsilon'=\frac{\varepsilon}{4}$ and $\alpha' = (\frac{\varepsilon}{2})^{\frac{1}{n}} \, \alpha$.
Consider a compact $(n-1)$-subcomplex~$Y \subseteq X$ which $\rho$-separates~$X$ and has minimal volume up to~$\delta$, where $\delta = \frac{\varepsilon'}{(n-1)!} \, \left( \prod_{i=1}^n \lambda_i \right) \, \alpha'^{n-1} \, (\inf \rho)^{n-1} >0$.
By Lemma~\ref{lem:XY}, the space~$Y$ does not $\rho_Y$-collapse.
Thus, by induction (replacing $\varepsilon$ with~$\varepsilon'$ and $\alpha$ with~$\alpha'$), there exists $y_0 \in Y$ such that
\begin{equation} \label{eq:volY}
\HH_{n-1}(B_Y(y_0,s)) \geq \frac{1-\varepsilon'}{(n-1)!}  \left( {\textstyle \prod\limits_{i=1}^{n-1}} \lambda_i \right)  s^{n-1}
\end{equation}
for every $s \in [\alpha'  \, \rho(y_0),\rho(y_0)]$.

Denote $x_0:=y_0$ and let $s \in [\alpha'  \rho(x_0),\rho(x_0))$.
Consider the compact $(n-1)$-subcomplex $Z \subseteq X$ obtained from~$Y$ by replacing $Y \cap B_X(y_0,s)$ with the sphere~$S_X(y_0,s)$.
The resulting space~$Z$ $\rho$-separates~$X$.
Indeed, the closure of the connected component~$B_X(x_0,s)$ of~$X \setminus Z$ is clearly covered by the ball $B_X(x_0,\rho(x_0))$.
Furthermore, all the other connected components of~$X \setminus Z$ are contained in connected components of~$X \setminus Y$ and therefore their closures also lie in balls $B_X(x,\rho(x))$ with $x \in X$.

Now, since $Y$ has minimal volume up to~$\delta$ among all compact $(n-1)$-subpolyhedra that $\rho$-separate~$X$, we derive
\[
\HH_{n-1}(Z) \geq \HH_{n-1}(Y) - \delta.
\]
This implies
\begin{align*}
\HH_{n-1}(S_X(x_0,s)) & \geq \HH_{n-1}(Y \cap B_X(y_0,s)) - \delta \\
 & \geq \HH_{n-1}(B_Y(y_0,s)) - \delta \\
 & \geq \frac{1-\varepsilon'}{(n-1)!}  \left( {\textstyle \prod\limits_{i=1}^{n-1}} \lambda_i \right)  s^{n-1} - \frac{\varepsilon'}{(n-1)!}  \left( {\textstyle \prod\limits_{i=1}^{n-1}} \lambda_i \right)  s^{n-1} \\
 & \geq \frac{1- 2 \varepsilon'}{(n-1)!}  \left( {\textstyle \prod\limits_{i=1}^{n-1}} \lambda_i \right)  s^{n-1}.
\end{align*}
where the second inequality follows from the inclusion $B_Y(y_0,s) \subseteq Y \cap B_X(y_0,s)$ and the third inequality follows from~\eqref{eq:volY} and the bound $s \geq \alpha' (\inf \rho)$.

Thus, by the Eilenberg inequality~\eqref{eq:eilenberg}, we obtain
\[
\HH_n(B_X(x_0,r)) \geq \lambda_n \, \int^{r \mkern-28mu *}_{\alpha'  \rho(x_0)} \HH_{n-1}(S_X(x_0,s)) \, ds \geq \frac{1- 2 \varepsilon'}{n!}  \left( {\textstyle \prod\limits_{i=1}^{n}} \lambda_i \right)  \left( r^n - \alpha'^n \, \rho(x_0)^n \right)
\]
for every $r \in [\alpha  \, \rho(x_0),\rho(x_0)]$. 
(Note that $r \geq \alpha' \rho(x_0)$ since $\alpha' \leq \alpha$.)

We want this expression to be greater or equal to~$\frac{1-\varepsilon}{n!} \left( \prod_{i=1}^{n} \lambda_i \right) r^n$ for every $r \in [\alpha  \, \rho(x_0),\rho(x_0)]$.
That is, we want
\[
(\varepsilon - 2 \varepsilon') r^n \geq \alpha'^n \rho(x_0)^n.
\]
Since $r \geq \alpha  \, \rho(x_0)$, it is enough to have
\[
(\varepsilon - 2 \varepsilon') \alpha^n \geq \alpha'^n.
\]
Replacing $\varepsilon'$ and~$\alpha'$ with their expressions in terms of~$\varepsilon$ and~$\alpha$, we observe that this inequality is satisfied, and in fact, is an equality.
Hence,
\begin{equation*} 
\HH_n(B_X(x_0,r)) \geq \frac{1-\varepsilon}{n!} \left( {\textstyle \prod\limits_{i=1}^{n}} \lambda_i \right)  r^n
\end{equation*}
for every $r \in [\alpha  \, \rho(x_0),\rho(x_0)]$.

\medskip

Take a decreasing sequence~$(\varepsilon_m)$ converging to zero and fix $\alpha=\varepsilon_m$.
There exists $x_m \in X$ such that the $n$-dimensional Hausdorff measure of~$B_X(x_m,s)$ satisfies a similar lower bound to~\eqref{eq:lowerX} for every $s \in [\varepsilon_m \, \rho(x_m), \rho(x_m)]$ with $\varepsilon=\varepsilon_m$.
By compactness of~$X$, we can assume that $(x_m)$ converges to a point~$x_0 \in X$.
Also, by compactness of~$X$ and continuity of~$\rho$, the function~$\rho$ is bounded.
Let $r \in [0,\rho(x_0)]$.
By the triangle inequality, we have 
\[
B_X(x_0,r) \supseteq B_X(x_m,r-|x_0 x_m|).
\]
Thus, for $m$ large enough, we obtain
\[
\HH_n(B_X(x_0,r)) \geq \frac{1-\varepsilon_m}{n!} \left( {\textstyle \prod\limits_{i=1}^{n}} \lambda_i \right) \min\{ r-|x_0 x_m| ,\rho(x_m) \}^n.
\]
Since $|x_0 x_m|$ goes to zero and $\rho$ is continuous, we obtain the desired lower bound for~$\HH_n(B_X(x_0,r))$.
\end{proof}


\begin{remark}
If $X$ is a compact $n$-polyhderon with a piecewise Riemannian metric (\eg, a closed Riemannian $n$-manifold), we can apply the coarea formula, see~\cite[13.4.2]{BZ}, instead of Eilenberg's inequality in the proof of Theorem~\ref{theo:x_0}.
In this case, we can get rid of the product $\prod_{i=1}^n \lambda_i$ in the expression of~$c_n$.
\end{remark}

\begin{remark}
The lower bound~\eqref{eq:lowerX} holds even if $\rho$ is not continuous, provided it is bounded away from zero.
\end{remark}

Let us introduce a couple of definitions: one of topological nature and the other one of geometrical nature.

\begin{definition}
Let $X$ be a connected compact $n$-polyhedral space and $\phi:\pi_1(X) \to G$ be a group homomorphism to a discrete group~$G$.
The space~$X$ is \emph{$\phi$-essential} if the classifying map $\Phi:M \to K(G,1)$ induced by~$\phi$ is not homotopic to a continuous map $X \to P \to K(G,1)$ which factors out through a simplicial $(n-1)$-complex~$P$.
\end{definition}

\begin{definition}
Let $X$ be a connected compact $n$-polyhedral length space and $\phi:\pi_1(X) \to G$ be a nontrivial group homomorphism to a discrete group~$G$.
The \emph{$\phi$-systolic function} of~$X$ is the function $\sys_\phi:X \to (0,\infty)$ defined as 
\[
\sys_\phi(X,x) = \inf \{ \length(\gamma) \mid \gamma \mbox{ is a noncontractible loop of } X \mbox{ based at } x \}.
\]
Clearly, the $\phi$-systolic function is continuous and even $2$-Lipschitz.
\end{definition}

We have the following non-collapsing result.

\begin{proposition} \label{prop:sys}
Let $X$ be a $\phi$-essential connected compact $n$-polyhedral length space, where $\phi:\pi_1(X) \to G$ is a nontrivial group homomorphism to a discrete group~$G$.
Then the space~$X$ does not $\frac{1}{2} \sys_\phi$-collapse.
\end{proposition}

\begin{proof}
We follow again an argument presented in~\cite{nab}. 
Suppose that $X$ $\frac{1}{2} \sys_\phi$-collapses.
Denote by $\pi:\hat{X} \to X$ the cover corresponding to the subgroup~$\ker \phi \lhd \pi_1(X)$.
By Proposition~\ref{prop:mult}, there exists a finite covering of~$M$ of multiplicity at most~$n$ by open subsets~$U_i$ with $\bar{U}_i \subseteq B(x_i,\frac{1}{2} \sys_\phi(x_i))$ for some $x_i \in X$.
Without loss of generality, we can assume that the open subsets are connected.
Every loop~$\gamma$ in~$\bar{U}_i$ is homotopic to the concatenation of loops of length less than~$\sys_\phi(x_i)$ based at~$x_i$.
(The loops are of the form $\overline{x_ip_s} \cup \overline{p_sp_{s+1}} \cup \overline{p_{s+1}x_i}$, where the points $\{ p_i \}$ are given by a sufficiently fine subdivision of~$\gamma$.)
By definition of the $\phi$-systolic function, each homomorphism $\pi_1(U_i) \to \pi_1(X) \to G$ is trivial, where the first homomorphism is induced by the inclusion $U_i \hookrightarrow X$.
It follows that the preimage~$\pi^{-1}(U_i)$ of each open set~$U_i$ decomposes as the disjoint union of open sets~$\alpha \cdot \hat{U}_i \subseteq \hat{X}$, where $\hat{U}_i$ is a lift of~$U_i$ in~$\hat{X}$ and $\alpha$ runs over $H=\pi_1(X)/\ker \phi$.
Consider the nerve~$N$ of the covering~$\{ U_i \}$ of~$X$ and the nerve~$\hat{N}$ of the covering~$\{ \alpha \cdot \hat{U}_i \}$ of~$\hat{X}$.
We have the following commutative diagram
\begin{center}
\begin{tikzcd}
\hat{X} \arrow{r} \arrow{d}
&\hat{N} \arrow{d} \\
X \arrow{r} &N
\end{tikzcd}
\end{center}
where the horizontal maps correspond to the natural maps from a space to its nerve and the vertical maps are the quotient maps under the natural free actions of the subgroup~$H \leqslant G$.
We can construct an equivariant map $\hat{N} \to \tilde{K}$ to the universal cover~$\tilde{K}$ of~$K=K(G,1)$.
By construction, this map passes to the quotient giving rise to a map $N \to K$ such that the composite $X \to N \to K$ is homotopic to the classifying map induced by~$\phi$.
Since the covering~$\{ U_i \}$ has multiplicity at most~$n$, the nerve of~$N$ is a simplicial complex of dimension at most~$n-1$.
Therefore, the space~$X$ is not $\phi$-essential.
\end{proof}

\section{Margulis' constant and volume of large balls} \label{sec:margulis}

In this section, we introduce the Margulis function and apply the results of the previous section to obtain an exponential lower bound on the maximal volume of balls in the universal cover of a metric with negative curvature, and more generally, of a metric space satisfying similar features.

\medskip

We will need the following definition of the Margulis function.

\begin{definition} \label{def:margulis}
Let $X$ be a compact $n$-polyhedral length space.
Consider a group homomorphism $\phi:\pi_1(X) \to G$ to a discrete group~$G$.
For every $x \in X$ and $\mu >0$, denote by
\[
\Gamma^\mu_{\phi,x} = \langle \phi([\gamma]) \in G \mid \gamma \mbox{ loop of } X \mbox{ based at } x \mbox{ with } \length(\gamma) \leq \mu \rangle
\]
the subgroup of~$G$ generated by the $\phi$-image of the homotopy classes of the loops of~$X$ based at~$x$ of length at most~$\mu$.

The \emph{Margulis function} of~$X$ is the function $\mu_\phi:X \to (0,\infty]$ defined as 
\[
\mu_\phi(x) = \sup \{ \mu \mid \Gamma^\mu_{\phi,x} \mbox{ has subexponential growth} \rangle.
\]

Clearly, the Margulis function is continuous and even $2$-Lipschitz.
With this definition, the subgroup~$ \Gamma^\mu_{\phi,x}$ has subexponential growth for every $\mu < \mu_\phi(x)$.
Note that if the image of~$\phi$ has exponential growth, the Margulis function is bounded by $2 \, \diam(X)$.
%
When $\phi$ is the identity homomorphism, we simply write~$\mu$ for~$\mu_\phi$.
\end{definition}

\begin{remark}
One could replace ``has subexponential growth" with ``is virtually nilpotent" in the definition of the Margulis function since we will be primarily interested in groups satisfying the Tits alternative; see~Definition~\ref{def:tits}.
\end{remark}

We have the following non-collapsing result.

\begin{proposition} \label{prop:simpl}
Every closed $n$-manifold~$M$ with nonzero simplicial volume does not $\frac{1}{2} \mu$-collapse.
\end{proposition}

\begin{proof}
Suppose that $M$ $\frac{1}{2} \mu$-collapses.
By Proposition~\ref{prop:mult}, there exists a finite covering of~$M$ of multiplicity at most~$n$ by open subsets~$U_i$ with $\bar{U}_i \subseteq B(x_i,\frac{1}{2} \mu(x_i))$ for some $x_i \in X$.
Without loss of generality, we can assume that the open subsets~$U_i$ are connected.
The image of~$\pi_1(U_i)$ under the $\pi_1$-homomorphism induced by the inclusion map $U_i \hookrightarrow M$ is generated by the homotopy classes of loops of length at most~$\mu(x_i)$ based at~$x_i$.
By definition of the Margulis function, see Definition~\ref{def:margulis}, this subgroup of~$\pi_1(X)$ has subexponential growth.
Thus, it is amenable.
By Gromov's vanishing simplicial volume theorem, see~\cite{gro82} or~\cite{iva}, it follows that $M$ has zero simplicial volume.
\end{proof}

\begin{remark}
More generally, every closed $n$-manifold with nonzero minimal volume entropy does not $\frac{1}{2} \mu$-collapse; see~\cite[Corollary~2.11]{BS}.
\end{remark}

Let us introduce a quantitative version of the Tits alternative for groups.

\begin{definition} \label{def:tits}
Let $\kappa$ be a positive integer.
A group~$G$ satisfies the \emph{$\kappa$-Tits alternative} if for every symmetric subset~$S$ of~$G$ containing the identity, either $S$ generates a subgroup of subexponential growth, or there exist two elements in~$S^\kappa$ generating a nonabelian free subgroup.
In the latter case, the subgroup generated by~$S$ has exponential growth and its algebraic entropy (or exponential growth rate) with respect to~$S$ is at least~$\frac{\log(3)}{\kappa}$.
\end{definition}

\begin{example} \label{ex}
Our main source of examples of groups satisfying the $\kappa$-Tits alternative is given by fundamental groups of closed Riemanian $n$-manifolds with sectional curvature lying between $-k^2$ and~$-1$ for some $k \geq 1$, where $\kappa=\kappa(n,k)$ depends only on~$n$ and~$k$; see~\cite{DKL19}.
\end{example}

We can now state the main general result of this section.

\begin{theorem} \label{theo:ball}
Let $X$ be a compact $n$-polyhedral length space and $\phi:\pi_1(X) \to G$ be a group homomorphism whose image satisfies the $\kappa$-Tits alternative.
Suppose that $X$ does not $\frac{1}{2} \mu_\phi$-collapse.
Then there exists $\tilde{x}_0 \in \tilde{X}$ such that 
\[
\HH_n(B_{\tilde{X}}(\tilde{x}_0,r)) \geq \vol_n(B_{\hyp^n}(C \, r))
\]
for every $r \geq 0$, where $C = C_{n,\kappa} \cdot \min\{ 1,\frac{1}{\sup \mu_\phi} \}$ with $C_{n,\kappa} >0$ depending only on~$n$ and~$\kappa$.
\end{theorem}

\begin{proof}
According to Theorem~\ref{theo:x_0}, there exists $x_0 \in X$ such that for every lift $\tilde{x}_0 \in \tilde{X}$ of~$x_0$,
\begin{equation} \label{eq:H>r^n}
\HH_n(B_{\tilde{X}}(\tilde{x}_0,r)) \geq \HH_n(B_X(x_0,r)) \geq c_n \, r^n
\end{equation}
for every $r \in [0, \frac{1}{2} \mu_\phi(x_0)]$, where $c_n = \frac{1}{n!} \, \prod_{i=1}^n \lambda_i$.
For simplicity, denote $\mu_0=\mu_\phi(x_0)$.

\medskip

Let us establish an exponential lower bound for the volume of balls of large enough radius centered at~$\tilde{x}_0$.

\begin{lemma}
For every $r \geq \frac{1}{2} \, \mu_0$, we have
\begin{equation} \label{eq:H>e^r}
\HH_n(B_{\tilde{X}}(\tilde{x}_0,r)) \geq \frac{c_n}{2^n} \, \mu_0^n \left( 3^{\left \lfloor \frac{r-\frac{\mu_0}{2}}{2 \kappa \mu_0} \right \rfloor } -1 \right)
\end{equation}
where $\lfloor \cdot \rfloor$ is the floor function.
\end{lemma}

\begin{proof}
The idea of the proof is to find a nonabelian free subgroup~$F$ of~$\pi_1(X,x_0)$ generated by loops of length at most~$2 \kappa \mu_0$ such that the translates of~$B_{\tilde{X}}(\tilde{x}_0,\frac{\mu_0}{2})$ under the free action of~$F$ on~$\tilde{X}$ are disjoint.

\medskip

Denote by~$S$ the symmetric generating set of~$\Gamma^{2 \mu_0}_{\phi,x_0} \leqslant G$ formed of the elements~$\phi([\gamma])$ where $\gamma$ is a loop of~$X$ based at~$x_0$ of length at most~$2 \mu_0$.
By definition of the Margulis function, see Definition~\ref{def:margulis}, the subgroup $\Gamma^{2 \mu_0}_{\phi,x_0} \leqslant G$ generated by~$S$ has exponential growth. 
(One could consider $\Gamma^{\mu_0+\varepsilon}_{\phi,x_0}$ instead of $\Gamma^{2 \mu_0}_{\phi,x_0}$ if one wanted to be more precise.)
Note also that every element of~$S^\kappa$ is the $\phi$-image of the homotopy class of a loop of~$X$ based at~$x_0$ of length at most $2 \kappa \mu_0$.
Thus, since the image of~$\phi$ in~$G$ satisfies the $\kappa$-Tits alternative, there exist two loops~$\gamma_1$ and~$\gamma_2$ based at~$x_0$ of length at most~$2 \kappa \mu_0$ whose homotopy classes~$\alpha_1$ and~$\alpha_2$ generate a nonabelian free subgroup of~$\pi_1(X,x_0)$.

\medskip

Consider a nonabelian free subgroup $F \leqslant \pi_1(X,x_0)$ generated by two (nontrivial) elements $\alpha_1$ and~$\alpha_2$ with $d_{\tilde{X}}(\tilde{x}_0,\alpha_i \cdot \tilde{x}_0) \leq 2 \kappa \mu_0$ such that the orbit $F \cdot \tilde{x}_0$ has a minimal number of points in~$B_{\tilde{X}}(\tilde{x}_0,\mu_0)$.
(The existence of such subgroup is ensured by the previous discussion.)
Suppose that the intersection $F \cdot \tilde{x}_0 \cap B_{\tilde{X}}(\tilde{x}_0,\mu_0)$ has a point~$\alpha \cdot \tilde{x}_0$ different from~$\tilde{x}_0$ with ~$\alpha \in F$.
Denote by~$k$ the smallest positive integer such that $\alpha^k \cdot \tilde{x}_0$ does not lie in $B_{\tilde{X}}(\tilde{x}_0,\mu_0)$.
We have
\begin{align*}
d_{\tilde{X}}(\tilde{x}_0,\alpha^k \cdot \tilde{x}_0) & \leq d_{\tilde{X}}(\tilde{x}_0,\alpha^{k-1} \cdot \tilde{x}_0) + d_{\tilde{X}}(\alpha^{k-1} \cdot \tilde{x}_0,\alpha^k \cdot \tilde{x}_0) \\
 & \leq \mu_0 + d_{\tilde{X}}(\tilde{x}_0,\alpha \cdot \tilde{x}_0) \\
 & \leq 2 \mu_0
\end{align*}
Thus, the displacement of~$\tilde{x}_0$ by~$\alpha^k$ is at most $2 \kappa \mu_0$.
(Note that $\kappa \geq 1$.)
Furthermore, either the subgroup generated by $\alpha_1$ and~$\alpha^k$, or the subgroup generated by $\alpha_2$ and~$\alpha^k$ is nonabelian free.
Denote by $F' \leqslant F$ this nonabelian free subgroup.
Since $k \geq 2$, the point $\alpha \cdot \tilde{x}_0$ does not lie in the orbit of~$F'$.
Therefore, the orbit~$F' \cdot \tilde{x}_0$ of~$F'$ has fewer elements in $B_{\tilde{X}}(\tilde{x}_0,\mu_0)$ than the orbit~$F \cdot \tilde{x}_0$ of~$F$, which contradicts the definition of~$F$.
Thus, the only point of the orbit~$F \cdot \tilde{x}_0$ lying in~$B_{\tilde{X}}(\tilde{x}_0,\mu_0)$ is~$\tilde{x}_0$.
This implies that the translates of~$B_{\tilde{X}}(\tilde{x}_0,\frac{\mu_0}{2})$ under the free action of~$F$ on~$\tilde{X}$ are disjoint.

\medskip

Let $r \geq \frac{1}{2} \, \mu_0$.
The ball $B_{\tilde{X}}(\tilde{x}_0,r)$ contains all the translates of $B_{\tilde{X}}(\tilde{x}_0,\frac{\mu_0}{2})$ under the elements of~$F =\langle \alpha_1, \alpha_2 \rangle$ which can be written as words of length at most 
\[
w=\left \lfloor \frac{r-\frac{\mu_0}{2}}{2 \kappa \mu_0} \right \rfloor
\]
with the letters~$\alpha_1$ and~$\alpha_2$.
Since the translates of $B_{\tilde{X}}(\tilde{x}_0,\frac{\mu_0}{2})$ under~$F$ are disjoint, we derive from~\eqref{eq:H>r^n} that
\[
\HH_n(B_{\tilde{X}}(\tilde{x}_0,r)) \geq c_n \, \frac{\mu_0^n}{2^n} \, (3^w-1)
\]
as desired.
\end{proof}

Let us resume the proof of Theorem~\ref{theo:ball} using the polynomial lower bound~\eqref{eq:H>r^n} and the exponential lower bound~\eqref{eq:H>e^r} on the volume of balls in the universal cover of~$X$.

Recall that
\begin{equation} \label{eq:volH}
\vol_n(B_{\hyp^n}(r)) = \sigma_{n-1} \, \int_0^r \sinh(t)^{n-1} \, dt
\end{equation}
where $\sigma_{n-1}$ is the Euclidean volume of the unit sphere~$S^{n-1}$.
We want to show that
\begin{equation} \label{eq:wanted}
\HH_n(B_{\tilde{X}}(\tilde{x}_0,r)) \geq \vol_n(B_{\hyp^n}(C \, r))
\end{equation}
for every $r \geq 0$, where $C>0$ is a metric-independent constant to determine.
We will consider the following situations:

\begin{enumerate}
\item If $r \geq 5 \kappa \mu_0$ then the estimate~\eqref{eq:H>e^r} yields the following lower bound
\begin{equation} \label{eq:r>5}
\HH_n(B_{\tilde{X}}(\tilde{x}_0,r)) \geq \frac{c_n}{2^n} \, \mu_0^n \left( 3^{\frac{r-\frac{\mu_0}{2}}{4\kappa \mu_0}}-1 \right).
\end{equation} 

\item If $r \leq 5 \kappa \mu_0$ then the estimate~\eqref{eq:H>r^n} yields the following lower bound
\begin{equation} \label{eq:r<5}
\HH_n(B_{\tilde{X}}(\tilde{x}_0,r)) \geq \frac{c_n}{2^n (5\kappa)^n} \, r^n.
\end{equation}

\item If $r \geq 1$ then the formula~\eqref{eq:volH} combined with the inequality $\sinh(t) \leq \frac{e^t}{2}$ for $t \geq 0$ yields the following upper bound
\begin{equation} \label{eq:r>1}
\vol_n(B_{\hyp^n}(C r)) \leq \frac{\sigma_{n-1}}{(n-1) 2^{n-1}} \, \left( e^{(n-1) C r} - 1 \right).
\end{equation}

\item If $r \leq 1$ then the formula~\eqref{eq:volH} combined with the inequality $\sinh(t) \geq \sinh(1) \, t$ for $0 \leq t \leq 1$ yields the following upper bound
\begin{equation} \label{eq:r<1}
\vol_n(B_{\hyp^n}(C r)) \leq \frac{\sigma_{n-1}}{n} \, \sinh(1)^{n-1} \, C^n \, r^n.
\end{equation}
\end{enumerate}

We will consider four cases analyzing the different combinations of situations, starting with the more technical cases.

\medskip

\noindent \underline{\it Case 1.}
Suppose that $r \geq 5 \kappa \mu_0$ and $r \geq 1$.
By~\eqref{eq:r>5} and~\eqref{eq:r>1}, the inequality~\eqref{eq:wanted} holds if
\begin{equation} \label{eq:r>5>1}
\frac{c_n}{2^n} \, \mu_0^n \left( 3^{\frac{r-\frac{\mu_0}{2}}{4\kappa \mu_0}}-1 \right) \geq \frac{\sigma_{n-1}}{(n-1) 2^{n-1}} \, \left( e^{(n-1) C r} - 1 \right).
\end{equation}
Comparing the exponential growth rates of the two sides of the inequality, we want $\frac{\log(3)}{4 \kappa \mu_0} \geq (n-1) C$.
That is, 
\begin{equation} \label{eq:log3}
C \leq \frac{\log(3)}{4(n-1) \kappa \mu_0}.
\end{equation}
Thus, we need to compare $f(t)=a \, e^{\alpha t} + a'$ and $g(t)= b \, e^{\beta t} +b'$ with $a, b >0$ and $\alpha \geq \beta >0$.
We have $f(t) \geq g(t)$ for every $t \geq t_0$ if $f(t_0) \geq g(t_0)$ and $f'(t_0) \geq g'(t_0)$.
In our case, we need to check the inequality~\eqref{eq:r>5>1} and the inequality
\[
\frac{c_n}{2^n} \, \mu_0^n \, \frac{\log(3)}{4 \kappa \mu_0} \, 3^{\frac{r-\frac{\mu_0}{2}}{4\kappa \mu_0}} \geq \frac{\sigma_{n-1}}{2^{n-1}} \, C \, e^{(n-1) C r}
\]
for $r = \max \{ 5 \kappa \mu_0, 1\}$.

If $5 \kappa \mu_0 \geq 1$, it is enough to have 
\[
C \leq \frac{1}{5 (n-1) \kappa \mu_0} \, \log \left( 1+ \frac{(n-1) c_n}{\sigma_{n-1}} \mu_0^n \right)
\]
for the first inequality (since the term between parenthesis in the left-hand side of~\eqref{eq:r>5>1} is bounded from below by~$2$ when $r=5 \kappa \mu_0$) and
\[
\frac{c_n}{2^n} \, \mu_0^n \, \frac{\log(3)}{4 \kappa \mu_0} \, 3^{-\frac{1}{8 \kappa}} \geq \frac{\sigma_{n-1}}{2^{n-1}} \, C
\]
for the second inequality using~\eqref{eq:log3}.
Since $\mu_0 \geq \frac{1}{5 \kappa}$, these conditions are satisfied if $C \leq \frac{C_{n,\kappa}}{\mu_0}$ for some constant $C_{n,\kappa} >0$ depending only on~$n$ and~$\kappa$.

If $5 \kappa \mu_0 \leq 1$, it is enough to have
\[
C \leq \frac{1}{n-1} \, \log \left( 1+ \frac{(n-1) c_n}{2 \sigma_{n-1}} \, ( 3^{\frac{1}{8 \kappa \mu_0}} - 1)  \mu_0^n \right)
\]
for the first inequality (since $r - \frac{\mu_0}{2} \geq \frac{r}{2}$ when $r=1$) and
\[
\frac{c_n}{2^n} \, \mu_0^n \, \frac{\log(3)}{4 \kappa \mu_0} \, 3^{\frac{1}{8 \kappa \mu_0}} \geq \frac{\sigma_{n-1}}{2^{n-1}} \, C \, e^{(n-1) C}
\]
for the second inequality.
Since $\mu_0 \leq \frac{1}{5\kappa}$, the expression $\mu_0^{n-1} \, 3^{\frac{1}{8 \kappa \mu_0}}$ is bounded from below by a positive constant depending only on~$n$ and~$\kappa$.
Thus, the previous conditions are satisfied if $C \leq C_{n,\kappa}$ for some constant $C_{n,\kappa} >0$ depending only on~$n$ and~$\kappa$.

\medskip

\noindent \underline{\it Case 2.}
Suppose that $r \geq 5 \kappa \mu_0$ and $r \leq 1$.
By~\eqref{eq:r>5} and~\eqref{eq:r<1}, the inequality~\eqref{eq:wanted} holds if
\begin{equation} \label{eq:r>5<1}
\frac{c_n}{2^n} \, \mu_0^n \left( 3^{\frac{r-\frac{\mu_0}{2}}{4\kappa \mu_0}}-1 \right) \geq \frac{\sigma_{n-1}}{n} \, \sinh(1)^{n-1} \, C^n \, r^n.
\end{equation}
Observe that the left-hand side of this inequality divided by~$r^n$ is bounded from below by a positive constant depending only on~$n$ and~$\kappa$ when $5 \kappa \mu_0 \leq r \leq 1$.
Thus, the condition~\eqref{eq:r>5<1} is satisfied if $C \leq C_{n,\kappa}$ for some constant $C_{n,\kappa} >0$ depending only on~$n$ and~$\kappa$.

\medskip

\noindent \underline{\it Case 3.}
Suppose that $r \leq 5 \kappa \mu_0$ and $r \geq 1$.
By~\eqref{eq:r<5} and~\eqref{eq:r>1}, the inequality~\eqref{eq:wanted} holds if
\begin{equation} \label{eq:r<5>1}
\frac{c_n}{2^n (5\kappa)^n} \, r^n \geq \frac{\sigma_{n-1}}{(n-1) 2^{n-1}} \, \left( e^{(n-1) C r} - 1 \right).
\end{equation}
Since $1 \leq r \leq 5 \kappa \mu_0$, the left-hand side of this inequality is bounded from below by $\frac{c_n}{2^n (5\kappa)^n}$ and the right-hand side is bounded from above by $\frac{\sigma_{n-1}}{(n-1) 2^{n-1}} \, \left( e^{5(n-1)\kappa \mu_0 \, C } - 1 \right)$.
Thus, the condition~\eqref{eq:r<5>1} is satisfied if $C \leq \frac{C_{n,\kappa}}{\mu_0}$ for some constant $C_{n,\kappa} >0$ depending only on~$n$ and~$\kappa$.

\medskip

\noindent \underline{\it Case 4.}
Suppose that $r \leq 5 \kappa \mu_0$ and $r \leq 1$.
By~\eqref{eq:r<5} and~\eqref{eq:r<1}, the inequality~\eqref{eq:wanted} holds if
\begin{equation} \label{eq:r<5<1}
\frac{c_n}{2^n (5\kappa)^n} \, r^n \geq \frac{\sigma_{n-1}}{n} \, \sinh(1)^{n-1} \, C^n \, r^n.
\end{equation}
This condition is satisfied if $C \leq C_{n,\kappa}$ for some constant $C_{n,\kappa} >0$ depending only on~$n$ and~$\kappa$.

\medskip

In conclusion, there exists a positive constant~$C_{n,\kappa}$ depending only on~$n$ and~$\kappa$ such that for every $C \leq C_{n,\kappa} \cdot \min\{ 1,\frac{1}{\mu_0} \}$
\begin{equation} \label{eq:end}
\HH_n(B_{\tilde{X}}(\tilde{x}_0,r)) \geq \vol_n(B_{\hyp^n}(C \, r))
\end{equation}
for every $r \geq 0$.
\end{proof}

\forget
\underline{\it Case 1.}
Suppose that $r \leq 1$.
By \eqref{eq:H>r^n} and~\eqref{eq:volH} combined with the inequality $\sinh(t) \leq \sinh(1) \, t$ for every $t \in [0,1]$, it is enough to show that
\[
c_n \, r^n \geq \frac{\sigma_{n-1}}{n} \, \sinh(1)^{n-1} \, C^n \, r^n.
\]
That is, $C \leq \left( \frac{n \, c_n}{\sigma_{n-1} \sinh(1)^{n-1}} \right)^\frac{1}{n}$.

\medskip

\underline{\it Case 2.}
Suppose that $r \geq 1$ and $r \geq \mu_0$.
By \eqref{eq:H>e^r} and~\eqref{eq:volH} combined with the inequality $\sinh(t) \leq \frac{e^t}{2}$ for $t \geq 0$, it is enough to show that
\begin{equation} \label{eq:ee}
\frac{c_n}{2^n} \, \mu_0^n  \left( 3^{\frac{r-\frac{\mu_0}{2}}{2 \kappa \mu_0}} - 1 \right) \geq \frac{\sigma_{n-1}}{(n-1) 2^{n-1}} \, \left( e^{(n-1) C r} - 1 \right)
\end{equation}
for every $r \geq \max\{ 1,\mu_0 \}$.
This requires that $\frac{\log(3)}{2 \kappa \mu_0} \geq (n-1) C$.
That is, $C \leq \frac{\log(3)}{2(n-1) \kappa \mu_0}$.
In this case, the relation~\eqref{eq:ee} only needs to be checked for $r \geq \max\{1,\mu_0\}$.

If $\mu_0 \geq 1$, this leads to the condition
\[
C \leq \min \left\{ \frac{\log(3)}{2(n-1) \kappa \mu_0}, \frac{1}{(n-1) \mu_0} \, \log \left( 1+ \frac{(n-1) c_n}{2 \sigma_{n-1}} \, ( 3^{\frac{1}{4\kappa}} - 1)  \mu_0^n \right) \right\}.
\]
Since $\mu_0 \geq 1$, this condition is satisfied if $C \leq C_{n,\kappa}$ for some constant $C_{n,\kappa} >0$ depending only on~$n$ and~$\kappa$.

If $\mu_0 \leq 1$, this leads to the condition
\[
C \leq \min \left\{ \frac{\log(3)}{2(n-1) \kappa \mu_0}, \frac{1}{(n-1)} \, \log \left( 1+ \frac{(n-1) c_n}{2 \sigma_{n-1}} \, ( 3^{\frac{1}{4 \kappa \mu_0}} - 1)  \mu_0^n \right) \right\}.
\]
Since $\mu_0 \leq 1$, this condition is satisfied if $C \leq \frac{C_{n,\kappa}}{\mu_0}$ for some constant $C_{n,\kappa} >0$ depending only on~$n$ and~$\kappa$.

\medskip

\underline{\it Case 3.}
Suppose that $r \geq 1$ and $r \leq \mu_0$.
By \eqref{eq:H>r^n} and~\eqref{eq:volH} combined with the inequality $\sinh(t) \leq \frac{e^t}{2}$ for $t \geq 0$, it is enough to show that
\[
c_n \, r^n \geq \frac{\sigma_{n-1}}{(n-1) 2^{n-1}} \, \left( e^{(n-1) C r} - 1 \right)
\]
for $1 \leq r \leq \mu_0$.
This just needs to be checked for $r=\mu_0$, for which we obtain the condition
\[
C \leq \frac{1}{(n-1) \mu_0} \, \log \left( 1+ \frac{(n-1) c_n}{\sigma_{n-1}} \, 2^{n-1}  \mu_0^n \right).
\]
Since $\mu_0 \geq 1$, this condition is satisfied if $C \leq \frac{C_{n}}{\mu_0}$ for some constant $C_{n} >0$ depending only on~$n$.

\medskip

In conclusion, there exists a constant~$C_{n,\kappa} >0$ depending only on~$n$ and~$\kappa$ such that for every $C \leq C_{n,\kappa} \cdot \min\{ 1,\frac{1}{\mu_0} \}$
\begin{equation} \label{eq:end}
\HH_n(B_{\tilde{X}}(\tilde{x}_0,r)) \geq \vol_n(B_{\hyp^n}(C \, r))
\end{equation}
for every $r \geq 0$.
\end{proof}
\forgotten

We immediately derive the following two corollaries.
The first corollary claims that when the volume of~$X$ is small enough, the constant~$C$ in Theorem~\ref{theo:ball} does not depend on the Margulis function, but only on~$n$ and~$\kappa$.

\begin{corollary} \label{coro:ball}
Let $X$ be a compact $n$-polyhedral length space and $\phi:\pi_1(X) \to G$ be a group homomorphism whose image satisfies the $\kappa$-Tits alternative.
Suppose that $X$ does not $\frac{1}{2} \mu_\phi$-collapse.
There exists a constant~$\delta_n>0$ such that if $\HH_n(X) \leq \delta_n$ then there exists $\tilde{x}_0 \in \tilde{X}$ such that 
\[
\HH_n(B_{\tilde{X}}(\tilde{x}_0,r)) \geq \vol_n(B_{\hyp^n}(C_{n,\kappa} \, r))
\]
for every $r \geq 0$, where $C_{n,\kappa} >0$ is a constant depending only on~$n$ and~$\kappa$.

In particular, there exists constant~$\delta'_{n,\kappa}>0$ such that if $\HH_n(X) \leq \delta'_{n,\kappa}$ then there exists $\tilde{x}_0 \in \tilde{X}$ such that 
\[
\HH_n(B_{\tilde{X}}(\tilde{x}_0,r)) \geq \vol_n(B_{\hyp^n}(r))
\]
for every $r \geq 1$.
\end{corollary}

\begin{proof}
Applying the second inequality of~\eqref{eq:H>r^n} to $r=\frac{1}{2} \mu_\phi(x_0)$, we obtain
\begin{equation} \label{eq:m0}
\mu_\phi(x_0)^n \leq  2^n c_n \, \HH_n(X).
\end{equation}
Thus, if $\HH_n(X) \leq \delta_n := \frac{1}{2^n c_n}$ then $\mu_\phi(x_0) \leq 1$ and we can take $C=C_{n,\kappa}$ in~\eqref{eq:end}, which proves the first part of the corollary.

Recall that the volume of a ball in~$\hyp^n$ grows exponentially with its radius.
Thus, we can choose~$\lambda=\lambda_{n,\kappa} >0$ large enough such that 
\[
\vol_n(B_{\hyp^n}(C_{n,\kappa} \, \lambda r)) \geq \lambda^n \, \vol_n(B_{\hyp^n}(r))
\]
for every $r \geq 1$.
Denote by $\lambda X$ the compact $n$-polyhedral length space obtained by multiplying the metric on~$X$ by~$\lambda$.
If $\delta'_{n,\kappa}=\frac{\delta_n}{\lambda^n}$ then $\HH_n(\lambda X) = \lambda^n \, \HH_n(X) \leq \delta_n$.
Therefore, there exists $\tilde{x}_0 \in \widetilde{\lambda X} = \lambda \tilde{X}$ such that 
\[
\HH_n(B_{\lambda \tilde{X}}(\tilde{x}_0,t)) \geq \vol_n(B_{\hyp^n}(C_{n,\kappa} \, t))
\]
for every $t \geq 0$.
Hence,
\[
\HH_n(B_{\tilde{X}}(\tilde{x}_0,r)) = \frac{1}{\lambda^n} \HH_n(B_{\lambda \widetilde{X}}(\tilde{x}_0,\lambda r)) \geq \frac{1}{\lambda^n} \vol_n(B_{\hyp^n}(C_{n,\kappa} \lambda r)) \geq \vol_n(B_{\hyp^n}(r))
\]
for every $r \geq 1$.
\end{proof}

\forget
\begin{corollary}
Let $X$ be a compact $n$-polyhedral length space and $\phi:\pi_1(X) \to G$ be a group homomorphism whose image satisfies the $\kappa$-Tits alternative.
Suppose that $X$ does not $\frac{1}{2} \mu_\phi$-collapse.
Then there exists $\tilde{x}_0 \in \tilde{X}$ such that 
\[
\HH_n(B_{\tilde{X}}(\tilde{x}_0,r)) \geq \vol_n(B_{\hyp^n}(r))
\]
for every $r \geq \HH_n(X)^{\frac{1}{n}}$.
\end{corollary}

\begin{proof}
We argue as in the proof of Corollary~\ref{theo:ball}.
By Theorem~\ref{theo:ball}, we have $r \geq \HH_n(X)^{\frac{1}{n}} \geq a_n \, \sup \mu_\phi$ for $a_n = \frac{1}{2 \sqrt[n]{c_n}}$; see~\eqref{eq:m0}.
In particular, $\min \{1,\frac{1}{\sup \mu_\phi} \} \, r$ is bounded away from zero.
Choose~$\lambda=\lambda_{n,\kappa} >0$ large enough such that 
\[
\vol_n(B_{\hyp^n}(C_{n,\kappa} \, {\textstyle \min \{1,\frac{1}{\sup \mu_\phi} \}} \, \lambda r)) \geq \lambda^n \, \vol_n(B_{\hyp^n}(r))
\]
for every $r \geq \HH_n(X)^{\frac{1}{n}}$.
Denote by $\lambda X$ the compact $n$-polyhedral length space obtained by multiplying the metric on~$X$ by~$\lambda$.
By Theorem~\ref{theo:ball}, there exists $\tilde{x}_0 \in \widetilde{\lambda X} = \lambda \tilde{X}$ such that 
\[
\HH_n(B_{\lambda \tilde{X}}(\tilde{x}_0,t)) \geq \vol_n(B_{\hyp^n}(C_{n,\kappa} \, {\textstyle \min \{1,\frac{1}{\sup \mu_\phi} \}} \, t))
\]
for every $t \geq 0$.
Hence,
\[
\HH_n(B_{\tilde{X}}(\tilde{x}_0,r)) = \frac{1}{\lambda^n} \HH_n(B_{\lambda \widetilde{X}}(\tilde{x}_0,\lambda r)) \geq \frac{1}{\lambda^n} \vol_n(B_{\hyp^n}(C_{n,\kappa} \, {\textstyle \min \{1,\frac{1}{\sup \mu_\phi} \}} \, \lambda r)) \geq \vol_n(B_{\hyp^n}(r))
\]
for every $r \geq \HH_n(X)^{\frac{1}{n}}$.
\end{proof}
\forgotten

The second corollary is what occurs when we apply the previous result to a closed manifold admitting a metric with negative sectional curvature.

\begin{corollary} \label{coro:final}
Let $M$ be a closed $n$-manifold admitting a metric with sectional curvature lying between $-k^2$ and~$-1$ for some $k \geq 1$.
Then for every Riemannian metric on~$M$ with $\vol(M) \leq \delta'_{n,\kappa}$ for some constant $\delta'_{n,\kappa}>0$ depending only on~$n$, there exists $\tilde{x}_0 \in \tilde{M}$ such that 
\begin{equation} \label{eq:C-M}
\vol(B_{\tilde{M}}(\tilde{x}_0,r)) \geq \vol(B_{\hyp^n}(r))
\end{equation}
for every $r \geq 1$.
\end{corollary}

\begin{proof}
As explained in Example~\ref{ex}, the fundamental group of~$M$ satisfies the $\kappa$-Tits alternative for some constant $\kappa=\kappa(n,k)$ depending only on~$n$ and~$k$; see~\cite{DKL19}.
Since $M$ is a closed manifold admitting a metric of negative curvature, its simplicial volume is nonzero; see~\cite{gro82}.
By Proposition~\ref{prop:simpl}, this implies that $M$ does not $\frac{1}{2} \mu$-collapse.
Therefore, the conclusion of Corollary~\ref{coro:ball} holds for~$M$.
\end{proof}

\end{document}